\documentclass[english, 10pt]{amsart}

\usepackage{tikz}
\usepackage{aliascnt}
\usepackage{amsfonts}
\usepackage{mathrsfs}
\usepackage{amsthm}
\usepackage{amsmath,amssymb}
\usepackage[all,poly,knot]{xy}
\usepackage{amsfonts}
\usepackage{color}
\usepackage{hyperref}
\usepackage{comment}  
\usepackage{csquotes}

\theoremstyle{plain}
\newtheorem{thmx}{Theorem} 

\newtheorem{thm}{Theorem}[section]  

\newtheorem{cor}[thm]{{Corollary}} 

\newtheorem{lem}[thm]{{Lemma}}

\newtheorem{prop}[thm]{Proposition}

\newtheorem{ques}[thm]{{Question}}

\newtheorem{defi}[thm]{Definition}

\theoremstyle{remark}
\newtheorem{rmk}[thm]{Remark}

\numberwithin{equation}{section}

\def\Spec{\mathrm{Spec}}

\def\log{\mathrm{log}\,}

\def\Spec{\mathrm{Spec}\,}

\usepackage[top=1.5in, bottom=1.5in, left=1in, right=1in]{geometry}

 \usepackage[hyperpageref]{backref}
 
\usepackage{xcolor}
\hypersetup{
	colorlinks,
	linkcolor={red!50!black},
	citecolor={blue!50!black},
	urlcolor={blue!80!black}
}
\begin{document} 
\title[Non-archimedean Borel Hyperbolicity]{Non-archimedean hyperbolicity of the moduli space of curves} 

	\author{Ruiran Sun}
		 \address{Institut fur Mathematik, Universit\"at Mainz, Mainz, 55099, Germany}
	\email{ruirasun@uni-mainz.de}

\begin{abstract}
Let $K$ be a complete algebraically closed non-archimedean valued field of characteristic zero, and let $X$ be a finite type scheme over $K$. We say $X$ is $K$-analytically Borel hyperbolic if, for every finite type reduced scheme $S$ over $K$, every rigid analytic morphism from the rigid analytification $S^{\mathrm{an}}$ of $S$ to the rigid analytification $X^{\mathrm{an}}$ of $X$ is algebraic. Using the Viehweg-Zuo construction and the $K$-analytic big Picard theorem of Cherry-Ru, we show that, for $N \geq 3$ and $g \geq 2$, the fine moduli space $\mathcal{M}^{[N]}_{g,K}$ over $K$ of genus $g$ curves with level $N$-structure is $K$-analytically Borel hyperbolic. 
\end{abstract}


\maketitle

\section{Introduction}
Motivated by conjectures of Green-Griffiths-Lang and higher dimensional generalizations of the Shafarevich problem, there has recently been much work on different notions of hyperbolicity \cite{Lan86, Java20} and the verification of them on the moduli spaces of smooth projective varieties \cite{KovacsLieblich, CP, JSZ}. In this paper, we study the non-archimedean analogue of Borel hyperbolicity introduced in \cite{Java-Kuch}, and verify it for $\mathcal{M}^{[N]}_{g,K}$, the moduli space of genus $g$ curves with level $N$-structure over a non-archimedean field $K$, where $g>1$ and $N \geq 3$.\\
Inspired by Cherry's work \cite{Cherry94}, in \cite{Java-Vezz} Javanpeykar and Vezzani introduced the non-archimedean analogue of Brody hyperbolicity. Let $X$ be a finite type scheme over $K$, where $K$ is a complete algebraically closed non-archimedean valued field of characteristic zero. Then $X$ is \emph{$K$-analytically Brody hyperbolic} if, for every finite type connected group scheme $G$ over $K$, every morphism $G^{\mathrm{an}} \to X^{\mathrm{an}}$ of rigid analytic varieties is constant, where $G^{\mathrm{an}}$ and $X^{\mathrm{an}}$ denote the rigid analytification of $G$ and $X$ respectively. In the  aforementioned paper they also proved that $\mathcal{A}^{[N]}_{g,K}$, the fine moduli space of $g$-dimensional principally polarized abelian schemes with level $N$-structure ( $N \geq 3$ ) over $K$, is $K$-analytically Brody hyperbolic if the residue field of $K$ has characteristic $0$. As a direct corollary of the Torelli theorem, the same statement holds for $\mathcal{M}^{[N]}_{g,K}$. In this paper we extend the latter result by proving the hyperbolicity of $\mathcal{M}^{[N]}_{g,K}$ for more general $K$ ( e.g., $K =\mathbb{C}_p$ ).
\begin{thmx}\label{Cor-B}
The variety $\mathcal{M}^{[N]}_{g,K}$ is $K$-analytically Brody hyperbolic for $N \geq 3$ and $g \geq 2$.  
\end{thmx}
There are notable differences between the complex case and non-archimedean case. For example, Berkovich {\cite[Theorem~4.5.1]{Ber90}} proved that Picard's Little Theorem holds for $\mathbb{G}_{m,K}$, i.e., there is no non-constant rigid analytic morphism from $\mathbb{A}^{1,\mathrm{an}}_K$ to $\mathbb{G}^{\mathrm{an}}_{m,K}$, which is contrary to the complex analytic case.
In fact this is the reason that Javanpeykar-Vezzani do not use the naive non-archimedean counterpart of the notion of complex Brody hyperbolicity in their aforementioned paper, since one does not want to regard $\mathbb{G}_{m,K}$ as a ``hyperbolic'' variety.\\ 
It is interesting to look for a suitable notion of non-archimedean hyperbolicity. As we have seen above,  $K$-analytically Brody hyperbolic varieties defined in \cite{Java-Vezz} are non-archimedean analogue of ``groupless varieties'' ( cf. \cite{JX19, JK20} ) and should therefore correspond to hyperbolic varieties by Lang's Conjecture {\cite[Conjecture~12.1]{Java20}}. In the aforementioned paper \cite{Java-Vezz}, they also suggested to pursue the non-archimedean analogue of Borel hyperbolicity, which we define as follows. 
\begin{defi}[Non-archimedean Borel hyperbolicity]\label{K-Borel-defi}
Let $K$ be a complete algebraically closed non-archimedean valued field of characteristic zero, and let $X$ be a finite type scheme over $K$.  We say $X$ is \emph{$K$-analytically Borel hyperbolic} if, for every finite type reduced scheme $S$ over $K$, any morphism $S^{\mathrm{an}} \to X^{\mathrm{an}}$ is algebraic.  
\end{defi}
Our main result is the following:
\begin{thmx}\label{thm-A}
If $g\geq 2$, $N\geq 3$, and $K$ is a complete algebraically closed non-archimedean valued field of characteristic zero, then $\mathcal{M}_{g,K}^{[N]}$ is $K$-analytically Borel hyperbolic.
\end{thmx}
{Theorem~\ref{thm-A} plus the characterization of $K$-analytic Brody hyperbolicity in {\cite[Theorem~2.18]{Java-Vezz}} imply Theorem~\ref{Cor-B}. Details of the proofs are given in Section~\ref{spreading-out}.}

Our proof of Theorem~\ref{thm-A} uses the Viehweg-Zuo construction associated to the universal family over $\mathcal{M}_{g,K}^{[N]}$ and the $K$-analytic big Picard theorem of Cherry-Ru (Theorem~\ref{extension}). 
Write the universal family as $V_K \to U_K$ for simplicity.
First we show a non-archimedean analogue of \cite[Theorem~1.5]{Java-Kuch}, which helps us to reduce the general case to the case of rigid analytic morphisms from curves. This is done in Section~\ref{curve-testing}. To show the algebraicity of a rigid analytic morphism $f$ from a curve $C_K$ to the base space $U_K$ of the family, we shall use the $K$-analytic big Picard theorem of Cherry-Ru to extend $f$ as a rigid analytic morphism between the compactifications of $C_K$ and $U_K$. Then the algebraicity follows from the rigid GAGA theorem.\\
As we mentioned, the $K$-analytic big Picard theorem of Cherry-Ru \cite{Cherry-Ru} plays a crucial role in this paper.
Cherry-Ru's theorem requires the existence of certain symmetric differentials of the base space $U_K$ which do not vanish along the rigid analytic morphism $f$. We shall use the Viehweg-Zuo construction associated to the family $V_K \to U_K$ to produce such a symmetric differential. {Basics of the Viehweg-Zuo construction are reviewed in Section~\ref{VZ-construction}}. However, the Viehweg-Zuo construction is designed for families over complex numbers. {So in Section~\ref{spreading-out}, we shall descend the original family to a family over a finitely generated subfield of $K$, base change it to get a complex family, apply the usual Viehweg-Zuo construction to this complex family and produce a symmetric differential on the complex base space, and use the descent argument again to obtain a symmetric differential on $U_K$. The final step is to check that all the requirements in the theorem of Cherry-Ru are satisfied, which concludes the proof.}\\
We should mention that before the paper of Javanpeykar-Vezzani there are works on non-archimedean hyperbolic geometry from the point of view of Nevanlinna theory, for instance, variant non-archimedean analogues of Picard theorems ( cf. \cite{Cherry94, Ru01, Cherry-Wang, Cherry-Ru} ).\\[.2cm] 
For families of abelian varieties, 
{one can also use the Viehweg-Zuo construction to produce symmetric differentials on the base spaces, but it requires more work to check that they do not vanish along given rigid analytic morphisms.}
It is natural to ask:
\begin{ques}
Is $\mathcal{A}^{[N]}_{g,K}$ $K$-analytically Borel hyperbolic for $N \geq 3$?  
\end{ques}
For the complex case, $\mathcal{A}^{[N]}_{g,\mathbb{C}}$ is known to be Borel hyperbolic for $N \geq 3$ by the work of Borel \cite{Borel} ( in fact the name ``Borel hyperbolicity'' comes from the algebraicity theorem of Borel on the arithmetic quotients of bounded symmetric domains ).\\[.3cm]
\noindent{\bf Acknowledgment.} This paper owes a tremendous debt to Ariyan Javanpeykar: for giving me the topic on the non-archimedean hyperbolicity of moduli spaces, for invaluable comments and suggestions he provided and for helping me to improve this paper. I am grateful to Alberto Vezzani for very helpful suggestions on Section~\ref{curve-testing}. I would like to thank Professor Kang Zuo for explaining his celebrated work with Viehweg to me, and for his constant supports and encouragements.

\section{Recollections about Viehweg-Zuo construction}\label{VZ-construction}
We first consider families of {smooth projective connected varieties} over complex numbers.
Let $\mathcal{M}_h$ be the stack of smooth proper polarized varieties with semi-ample canonical divisor and Hilbert polynomial $h$ {over $\mathbb{Q}$}. Let $U$ be a smooth quasi-projective variety over $\mathbb{C}$. Let $\varphi:\, U \to \mathcal{M}_h \otimes \mathbb{C}$ be a morphism of stacks which is generically finite onto its image. {Note that $\varphi$ is the classifying map of a smooth family $V \to U$ of polarized varieties.} We denote by $n$ the fiber dimension, so that $n=\mathrm{deg}\,h$.\\
The next step is to compactify the family. First we find the smooth compactifications $U \subset Y$ and $V \subset X$ with the simple normal crossing boundary divisors $S := Y \setminus U$ and $\Delta:=X \setminus V$. Then $\pi:\, V \to U$ extends to a projective morphism $g:\, X \to Y$. We use the same notation $g$ to denote the induced log morphism $(X,\Delta) \to (Y,S)$.
After desingularization process we can assume that $g$ is a log smooth morphism over a big open subset containing $U$ (a Zariski open subset is said to be big if its complement has codimension at least $2$). After leaving out some codimension-2 subvarieties, we still use the notation $g:\,(X,\Delta) \to (Y,S)$, which is a partial compactification of $\pi:\, V \to U$.\\
Viehweg-Zuo construct a graded Higgs bundle $(F,\tau)$ which has a close relation with the deformation theory of the family (see {\cite[\S6]{VZ-1}} and {\cite[\S4]{VZ-2}} for details). Recall that $(F,\tau)$ has a bi-graded structure $(\bigoplus_{p+q=n}F^{p,q}, \bigoplus_{p+q=n}\tau^{p,q})$, where
\[
F^{p,q} := R^qg_*T^q_{X/Y}(-\log \Delta) /_{\mathrm{torsion}},
\]
and the component of the Higgs map $\tau^{p,q}$ is induced by the edge morphism of a long exact sequence of higher direct image sheaves
\[
\tau^{p,q}:\, R^qg_*T^q_{X/Y}(-\log \Delta) \to \Omega^1_Y(\log S) \otimes R^{q+1}g_*T^{q+1}_{X/Y}(-\log \Delta).
\]
We extend $F^{p,q}$ to the compactification by taking the reflexive hull. The Higgs maps $\tau^{p,q}$ extend automatically over codimension-2 subvarieties. Hereafter we still use $Y$ to denote the compactification of $U$, and $(F,\tau)$ is the graded Higgs bundle over $Y$.\\
One can iterate the Higgs map
\[
F^{n,0} \cong \mathcal{O}_Y \xrightarrow{\tau^{n,0}} \Omega^1_Y(\log S) \otimes F^{n-1,1} \xrightarrow{\mathrm{Id}\otimes \tau^{n-1,1}} \Omega^1_Y(\log S)^{\otimes 2} \otimes F^{n-2,2} \to \cdots
\]
and obtain $\mathcal{O}_Y \to \Omega^1_Y(\log S)^{\otimes k} \otimes F^{n-k,k}$ for each $k$-th iteration. By the integrable condition of the Higgs map $\tau$, we have the factorization $\mathcal{O}_Y \to \mathrm{Sym}^k \Omega^1_Y(\log S) \otimes F^{n-k,k}$, which induces the map 
\[
\tau^k:\, \mathrm{Sym}^kT_Y(-\log S) \to F^{n-k,k}.
\]
For $k=1$ it becomes
\[
\tau^1:\, T_Y(-\log S) \to R^1g_*T_{X/Y}(-\log \Delta)
\]
which is exactly the (log) Kodaira-Spencer map associated to the family $g$.\\
For the given family $g:\, (X,\Delta) \to (Y,S)$ we can find a positive integer $m$, which is called the \emph{Griffiths-Yukawa coupling length} of the family $g$, such that the $m$-th iteration $\tau^m$ is a nonzero map and it factors through
\begin{align}\label{tau-m}
\tau^m:\, \mathrm{Sym}^mT_Y(-\log S) \to N^{n-m,m}
\end{align}
where $N^{n-m,m}:= \mathrm{Ker}(\tau^{n-m,m})$. Taking the image of the dual of $\tau^m$, we can obtain a subsheaf $\mathcal{P}$ of $\mathrm{Sym}^m\Omega^1_Y(\log S)$. 
Recall that Viehweg defined the bigness for torsion-free sheaves. See Definition~1.1 and Lemma~1.2 in \cite{VZ-2} for details.
For this subsheaf $\mathcal{P}$ Viehweg-Zuo proved the following important result:
\begin{thm}[Viehweg-Zuo, {\cite[Theorem~1.4]{VZ-2}}]\label{VZ-big-sheaf}
Let $V \to U$ be a smooth family of polarized varieties with semi-ample canonical divisor. Suppose that the family has maximal variation, i.e. the induced classifying map from $U$ to the moduli space is generically finite onto its image. 
Then the subsheaf $\mathcal{P}$ constructed above is big in the sense of Viehweg. In particular, one can find an ample invertible sheaf $\mathcal{H}$, some positive integer $\eta$ and a morphism 
\[
\bigoplus \mathcal{H} \longrightarrow \mathrm{Sym}^{\eta} \mathcal{P}
\]
which is surjective over some Zariski open subset of $U$.  
\end{thm}
The subsheaf $\mathcal{P}$ is commonly referred as the \emph{Viehweg-Zuo big subsheaf}.

\section{Testing $K$-analytic Borel hyperbolicity on maps from curves}\label{curve-testing}
In this section we shall prove a non-archimedean analogue of Theorem~1.5 in \cite{Java-Kuch}. {Hereafter we denote $X^{\mathrm{an}}$ as the rigid analytification of a finite type $K$-scheme $X$ {\cite[\S5.4]{Bosch}}.}
\begin{thm}[Testing $K$-analytic Borel hyperbolicity on maps from curves]\label{K-testing}
Let $K$ be an algebraically closed field of characteristic $0$ which is complete with respect to some non-archimedean valuation.
Let $X$ be a finite type separated scheme over $K$. Then the following are equivalent.
\begin{itemize}
\item[(i)] $X$ is $K$-analytically Borel hyperbolic (Definition~\ref{K-Borel-defi}).
\item[(ii)] For every smooth connected algebraic curve $C$ over $K$, every rigid analytic morphism $C^{\mathrm{an}} \to X^{\mathrm{an}}$ is algebraic. 
\end{itemize}
\end{thm}
In the rest of this section we will prove the non-archimedean counterpart of the complex-analytic results in section~2.2 of \cite{Java-Kuch}. After establishing those results about the rigid analytification, we shall give a proof of Theorem~\ref{K-testing} following the line of reasoning in \cite{Java-Kuch}.
\subsection{Some facts about the rigid analytification}
Let $X$ be a finite type $K$-scheme and $X^{\mathrm{an}}$ be its rigid analytification. We first study the relation between the ring of regular functions on $X$ and the ring of analytic functions on $X^{\mathrm{an}}$.
\begin{prop}[Analogue of Proposition~2.2 in \cite{Java-Kuch}]\label{prop-2.2}
If $X$ is a finite type integral scheme over $K$ of pure dimension, then the ring $\mathcal{O}(X)$ of global regular functions is integrally closed in
 the ring $\mathcal{H}(X^{\mathrm{an}})$ of global analytic functions.  
\end{prop}
\begin{proof}
To check the integrality we can localize to the case where $X$ is affine.
Write $X=\Spec A$, so that $\mathcal{O}(X)=A$. Then $A \subset \mathcal{H}(X^{\mathrm{an}})$ is a subring.
Now let us consider an element $f \in \mathcal{H}(X^{\mathrm{an}})$ which is integral over $A$. We set $B=A[f]$. The goal is to show that $B=A$.\\[.2cm]
Let $Y = \Spec B$.
So we have a finite morphism $\pi:\, Y \to X$ between $K$-schemes. Since $X$ is irreducible and $B=A[f]$, we know that $Y$ is also irreducible.\\
We now consider the rigid analytification of $\pi:\,Y \to X$. 
We have the following diagram
\[
\xymatrix{
Y^{\mathrm{an}} \ar[r] \ar[d]_{\pi^{\mathrm{an}}} & Y \ar[d]_{\pi} \\
X^{\mathrm{an}} \ar[r] \ar[ru]^-s  & X.
}
\]
where $s:\, X^{\mathrm{an}} \to Y$ is the map between locally ringed spaces induced by the inclusion $B \hookrightarrow \mathcal{H}(X^{\mathrm{an}})$. By the universal property of the rigid analytification functor, we know that $s$ factors through $s^{\mathrm{an}}:\, X^{\mathrm{an}} \to Y^{\mathrm{an}}$, which is an analytic section of $\pi^{\mathrm{an}}$
\[
\xymatrix{
Y^{\mathrm{an}} \ar[d]^{\pi^{\mathrm{an}}} \\
X^{\mathrm{an}} \ar@/^1pc/[u]^{s^{\mathrm{an}}}
}
\]
We claim that $s^{\mathrm{an}}$ is a Zariski closed immersion of rigid analytic varieties. To see this, we only need to notice that $(\pi^{\mathrm{an}})^*:\, \mathcal{H}(X^{\mathrm{an}}) \to \mathcal{H}(Y^{\mathrm{an}})$ is injective and the following composition
\[
\mathcal{H}(X^{\mathrm{an}}) \xrightarrow{(\pi^{\mathrm{an}})^*} \mathcal{H}(Y^{\mathrm{an}}) \xrightarrow{(s^{\mathrm{an}})^*} \mathcal{H}(X^{\mathrm{an}})
\]
is the identity map ( here we use the fact that $s^{\mathrm{an}}$ is a section of $\pi^{\mathrm{an}}$ ). So $(s^{\mathrm{an}})^*$ is surjective and thus $s^{\mathrm{an}}$ is a Zariski closed immersion.\\ 
On the other hand, we know that $\mathrm{dim}\,Y^{\mathrm{an}} = \mathrm{dim}\,X^{\mathrm{an}}$ as $\pi^{\mathrm{an}}$ is finite. Both $X^{\mathrm{an}}$ and $Y^{\mathrm{an}}$ are irreducible since they are the analytification of irreducible $K$-schemes.  Thus $Y^{\mathrm{an}} = X^{\mathrm{an}}$. We conclude that $B=A$ and $f \in A$.
\end{proof}

\begin{prop}[Analogue of Proposition~2.3 in \cite{Java-Kuch}]\label{prop-2.3}
Let $\mathcal{X}$ be a normal rigid analytic space; let $\mathcal{A} \subset \mathcal{X}$ be a proper closed analytic subset. Then the ring $\mathcal{H}(\mathcal{X})$ is integrally closed in $\mathcal{H}(\mathcal{X} \setminus \mathcal{A})$.  
\end{prop}
\begin{proof}
Let $f$ be an analytic function in $\mathcal{H}(\mathcal{X} \setminus \mathcal{A})$ and is integral over $\mathcal{H}(\mathcal{X})$, namely one can find $a_i \in \mathcal{H}(\mathcal{X})$ such that
\[
f^d + a_{d-1}f^{d-1} + \cdots + a_0 = 0.
\]
Then around each point of $\mathcal{A}$ the $a_i$ are bounded, hence so is $f$, and by the Hebbarkeitssatz (see {\cite[p. 502]{Con99}} or {\cite[Theorem~1.6]{Lut74}}) it can be extended to an analytic function on all of $\mathcal{X}$.  
\end{proof}
\begin{cor}[Analogue of Corollary~2.5 in \cite{Java-Kuch}]\label{int_close}
Let $X$ be a normal integral finite type scheme over $K$ and let $\mathcal{A} \subset X^{\mathrm{an}}$ be a proper closed analytic subset. Then the ring of regular functions $\mathcal{O}(X)$ is integrally closed in the ring of analytic functions $\mathcal{H}(X^{\mathrm{an}} \setminus \mathcal{A})$.  
\end{cor}
\begin{proof}
Since $X$ is normal, it follows that $X^{\mathrm{an}}$ is normal. Therefore, the statement follows from Proposition~\ref{prop-2.2} and Proposition~\ref{prop-2.3}.  
\end{proof}
\subsection{Specialization lemma for power series}
The proof of Theorem~1.5 in \cite{Java-Kuch} needs a ``transcendental" specialization lemma for power series {\cite[Lemma~2.7]{Java-Kuch}}. We state it here in our situation.\\ 
Let $k \subset K$ be an algebraically closed subfield such that $K$ has
 infinite transcendence degree over $k$. 
Then for some chosen $\lambda_1,\dots,\lambda_n \in K$ which are algebraically independent over $k$, we can define the following ring homomorphism
\[
\iota=\iota_{\lambda_1,\dots,\lambda_n}:\, k[x_1,\dots,x_{n+1}] \to K[z_1,\dots,z_n]
\]
by letting $\iota|_k$ be the inclusion $k \hookrightarrow K$, sending $x_j$ to $z_j$ for $1 \leq j \leq n$, and sending $x_{n+1}$ to the linear polynomial $\lambda_1z_1 + \cdots + \lambda_nz_n$. This homomorphism extends naturally to
\[
\iota:\, k[\![x_1,\dots,x_{n+1}]\!] \to K[\![z_1,\dots,z_n]\!].
\]
\begin{lem}[Lemma~2.7 in \cite{Java-Kuch}]\label{spec_lem}
Let $g \in k[\![x_1,\dots,x_{n+1}]\!]$. If $\iota(g) \in K[\![z_1,\dots,z_n]\!]$ is an algebraic function (i.e. when interpreted as an element of the quotient field $K(\!(z_1,\dots,z_n)\!)$ it is algebraic over the subfield $K(z_1,\dots,z_n)$ ), then $g$ is an algebraic function (i.e. $g$ is an element of $k(\!(z_1,\dots,z_{n+1})\!)$ algebraic over $k(z_1,\dots,z_{n+1})$).  
\end{lem}
The proof of {\cite[Lemma~2.7]{Java-Kuch}} in fact works for any infinite transcendence degree field extension $k \subset K$. So we omit the proof here and refer the reader to their paper.

\subsection{Dimension reduction}
The proof of Theorem~1.5 in \cite{Java-Kuch} is by induction on the dimension of the source spaces. So in this subsection we first show a non-archimedean counterpart of Proposition~3.6 (Dimension Lemma) in \cite{Java-Kuch}:
\begin{prop}[Dimension Lemma]\label{dim-lem}
Let $V$ and $X$ be algebraic varieties defined over $K$, where $V$ is normal and has dimension at least two, and let $f:\, V^{\mathrm{an}} \to X^{\mathrm{an}}$ be a rigid  analytic map. Suppose that for every closed algebraic subvariety $H \subset V$ of codimension one, the composition
\[
\tilde{H}^{\mathrm{an}} \xrightarrow{\nu^{\mathrm{an}}} H^{\mathrm{an}} \hookrightarrow V^{\mathrm{an}} \xrightarrow{f} X^{\mathrm{an}}
\]
is an algebraic morphism, where $\nu$ is the normalization of schemes. Then $f$ itself is algebraic.  
\end{prop}
\begin{proof}
We first choose a Zariski open subset $U \subset X$ which admits an embedding $j$ to some affine space $\mathbb{A}^m$. Denote by $\mathcal{A}$ the pull back of the complement $X \setminus U$ via the rigid analytic map $f$. Then we get analytic functions
\[
g=(g_1,\dots,g_m):\, V^{\mathrm{an}} \setminus \mathcal{A} \xrightarrow{f} U^{\mathrm{an}} \xrightarrow{j^{\mathrm{an}}} (\mathbb{A}^m_K)^{\mathrm{an}}.
\]
The goal is to show that all $g_i$'s are algebraic, i.e. that $g_i$'s are rational functions on $V$. To use the algebraicity assumption in Proposition~\ref{dim-lem}, we shall choose a suitable subvariety $H \subset V$ of codimension one.\\
We consider the Noether normalization $\pi:\, V \to \mathbb{A}^{n+1}$ (here $n+1 = \mathrm{dim}\, V>1$). One can choose a countable algebraically closed subfield $k \subset K$ so that $\pi$ is defined over $k$. Next we define the ``$k$-generic hyperplane" of $\mathbb{A}^{n+1}_K$
\[
P := \{ (z_1,\dots,z_{n+1})\in \mathbb{A}^{n+1}_K\,\,|\,\, z_{n+1} = \lambda_1z_1 +\cdots+\lambda_nz_n \}
\]
for $\lambda_i \in K$ algebraically independent over $k$. Then the induced homomorphism between complete local rings $\hat{\mathcal{O}}_{\mathbb{A}^{n+1}_k,0} \to \hat{\mathcal{O}}_{P,0}$ is exactly the map $\iota=\iota_{\lambda_1,\dots,\lambda_{n}}$ studied in Lemma~\ref{spec_lem}.\\
Now we choose $H \subset V$ to be some irreducible component of $\pi^{-1}(P)$, which has codimension one in $V$. Denote by $\tilde{H}$ the normalization. By our  algebraicity assumption, the restriction of $g_i$'s on $\tilde{H}^{\mathrm{an}} \setminus \mathcal{A}$, which we denote by $h_i$'s, are in fact algebraic.\\
We can choose $\pi$ at the beginning such that it is \'etale over $0 \in \mathbb{A}^{n+1}$ (and choose a preimage $\tilde{0}$ of $0$ such that $\tilde{0} \notin \mathcal{A}$). Then we have the following commutative diagram
\[
\xymatrix{
(V_k,\tilde{0}) \ar[d]^{\pi} & (\tilde{H},\tilde{0}) \ar[l] \ar[d]^{\pi|_{\tilde{H}}} \\
(\mathbb{A}^{n+1}_k,0) & (P,0) \ar[l]
}
\]
which induces the following homomorphisms of complete local rings
\[
\xymatrix{
\hat{\mathcal{O}}_{V_k,\tilde{0}} \ar[r]  & \hat{\mathcal{O}}_{\tilde{H},\tilde{0}}\\
k[\![x_1,\dots,x_{n+1}]\!] \ar[r]^{\iota} \ar[u]^{\pi^*}_{\cong} & K[\![z_1,\dots,z_n]\!] \ar[u]^{\cong}_{\pi^*}.
}
\]
Note that via $\pi^*$ we can identify the germs of $g_i$ in $\hat{\mathcal{O}}_{V_k,\tilde{0}}$ as formal power series in $k[\![x_1,\dots,x_{n+1}]\!]$. Here we use the property of the analytification functor $\hat{\mathcal{O}}_{V,\tilde{0}} \cong \hat{\mathcal{O}}_{V^{\mathrm{an}},\tilde{0}}$.
The image $\iota(g_i) \in K[\![z_1,\dots,z_n]\!]$ exactly corresponds to the germ of $h_i$, which is algebraic by our assumption. Thus by Lemma~\ref{spec_lem} we know that $g_i$ is integral over $k(x_1,\dots,x_{n+1})$. After shrinking $V$ to some Zariski open subset, we can assume that the finite morphism $\pi:\, V \to \mathbb{A}^{n+1}$ is \'etale onto its image, and thus the integrality of $g_i$ over $k(x_1,\dots,x_{n+1})$ implies that $g_i$ is integral over $\mathcal{O}(V_k)$ as an element in $\mathcal{H}(V^{\mathrm{an}} \setminus \mathcal{A})$.
From Corollary~\ref{int_close} we know that $\mathcal{O}(V_k)$ is integrally closed and thus $g_i$ is a regular function on $V$.
\end{proof}

\begin{proof}[Proof of Theorem~\ref{K-testing}]
We only need to show that $(\mathrm{ii}) \Longrightarrow (\mathrm{i})$.
Using the Dimension Lemma (Proposition~\ref{dim-lem}) and the induction on the dimension of the source space, one easily obtains the $K$-analytic Borel hyperbolicity of $X$.
\end{proof}

\section{Tools from non-archimedean Nevanlinna theory}
To show Borel hyperbolicity one needs certain extension theorem for analytic maps. We shall recall here the rigid big Picard theorem of Cherry-Ru, which can be regarded as a non-archimedean counterpart of Lu's extension theorem \cite[\S4, Lemma~3]{Lu-91}.\\ 
Let $K$ be an algebraically closed field of characteristic $0$ which is complete with respect to some {non-trivial} non-archimedean valuation $|\cdot|_K$. Typical examples of $K$ are $\mathbb{C}_p$ and $\overline{\mathbb{C} (\!(t)\!)}$. Denote by $A^1[r_1,r_2):= \{z \in K\,:\, r_1 \leq |z|_K <r_2 \}$ the annulus.
\begin{thm}[Cherry-Ru {\cite[Theorem~6.1]{Cherry-Ru}}]\label{extension}
Let $X$ be a smooth projective variety over $K$. Let $D$ be a simple normal crossing divisor on $X$. Let $f:\, A^1(0,R] \to X^{\mathrm{an}} \setminus D^{\mathrm{an}}$ be a rigid analytic morphism. If there exists a section $\omega$ in $H^0(X, \Omega^1_X(\log D)^{\otimes s})$ for some natural number $s$ such that $f^*\omega \not\equiv 0$ and $\omega$ vanishes along an ample divisor $A$ on $X$ (so $s \geq 1$), then $f$ extends to a rigid analytic morphism from $A^1[0,R]$ to $X^{\mathrm{an}}$.   
\end{thm}
We shall use the map (\ref{tau-m}) to produce some symmetric differential vanishing along an ample divisor, as required in Cherry-Ru's extension theorem. 

\section{From complex to $K$-coefficients}\label{spreading-out}
Let $K$ be a complete algebraically closed non-archimedean valued field of characteristic zero. Let $U_K$ be a smooth quasi-projective variety over $K$ which carries a morphism of stacks $U_K \to \mathcal{M}_h\otimes K$ generically finite onto its image. We denote by $\pi_K:\,V_K \to U_K$ the induced family of polarized varieties with maximal variation of moduli. 
We shall study the algebraicity of a rigid analytic morphism $f$ from $S^{\mathrm{an}}$, the rigid analytification of a finite type reduced scheme over $K$, to $U^{\mathrm{an}}_K$.\\
By Theorem~\ref{K-testing}, it suffices to study these rigid analytic morphisms $f$ whose source spaces $C^{\mathrm{an}}_K$ are curves.
To verify the algebraicity, we need to show that $f$ extends to a rigid analytic morphism between compactified spaces $\bar{C}^{\mathrm{an}}_K$ and $\bar{U}^{\mathrm{an}}_K$, and then apply the rigid GAGA theorem.\\[.2cm]
The first step is to construct the graded Higgs bundle used in the  Viehweg-Zuo construction.
Like families over complex numbers, we find smooth compactification $g_K:\, (X_K,\Delta_K) \to (Y_K,S_K)$ of $\pi_K:\, V_K \to U_K$. We define the graded Higgs bundle $(F_K,\tau_K)$ in the same manner (cf. Section~\ref{VZ-construction}). Note that all these constructions are purely algebraic, which is independent of the field of definition. We also have the iteration of Higgs maps
\[
\tau^k_K:\, \mathrm{Sym}^kT_{Y_K}(-\log S_K) \to F^{n-k,k}_K
\]
for $k=1,2,\dots,n$.\\

{The next step is to descend everything to a finitely generated subfield of $K$.
Notice that $X_K, Y_K$ are finite type $K$-schemes, the sheaves $F^{p,q}$ are coherent $\mathcal{O}_{Y_K}$-sheaves, and $\tau^{p,q}_K$ are morphisms between them. Therefore one can find a subfield $L \subset K$ which is finitely generated over $\mathbb{Q}$, and an $L$-model $(\mathcal{X},\Delta_{L}) \to (\mathcal{Y},\mathcal{S})$ of the family, as well as the $L$-model $(\mathcal{F},\tau)$ of the graded Higgs bundle.
Since $L$ is finitely generated over $\mathbb{Q}$, it can be embedded into $\mathbb{C}$ as a subfield.
By base change we obtain a family $g_{\mathbb{C}}:\, (X_{\mathbb{C}},\Delta_{\mathbb{C}}) \to (Y_{\mathbb{C}},S_{\mathbb{C}})$ over complex numbers, the graded Higgs bundle $(F_{\mathbb{C}},\tau_{\mathbb{C}})$, as well as the iteration of Higgs maps
\[
\tau^k_{\mathbb{C}}:\, \mathrm{Sym}^kT_{Y_{\mathbb{C}}}(-\log S_{\mathbb{C}}) \to F^{n-k,k}_{\mathbb{C}}.
\]
Now we apply the result of Viehweg-Zuo. Notice that the classifying map $\varphi_{\mathbb{C}}:\, U_{\mathbb{C}} \to \mathcal{M}_h \otimes \mathbb{C}$ is generically finite onto its image since its $L$-model is.\\
Let $m$ be the Griffiths-Yukawa coupling length of the complex family $g_{\mathbb{C}}$ with maximal variation. Then by Theorem~\ref{VZ-big-sheaf}, $\tau^m_{\mathbb{C}}$ factors through
\[
\mathrm{Sym}^mT_{Y_{\mathbb{C}}}(-\log S_{\mathbb{C}}) \to \mathcal{P}^{\vee}_{\mathbb{C}},
\]
where $\mathcal{P}_{\mathbb{C}}$ is big in the sense of Viehweg. Thus one can find an ample invertible sheaf $\mathcal{H}_{\mathbb{C}}$, some positive integer $\eta$ and a morphism 
\[
\mathrm{Sym}^{\eta} \mathcal{P}^{\vee}_{\mathbb{C}}
 \longrightarrow \bigoplus \mathcal{H}^{\vee}_{\mathbb{C}}
\]
which is injective over some Zariski open subset. The composed map $\mathrm{Sym}^{m\eta}T_{Y_{\mathbb{C}}}(-\log S_{\mathbb{C}}) \to \bigoplus \mathcal{H}^{\vee}_{\mathbb{C}}$ gives us plenty of symmetric differentials vanishing along some ample divisor, just as required in Theorem~\ref{extension}. But before applying Cherry-Ru's extension theorem we have to first ``transform'' those symmetric differentials back to the non-archimedean field $K$.\\
Note that although Viehweg-Zuo used some transcendental methods from Hodge theory to derive the bigness of $\mathcal{P}_{\mathbb{C}}$, the objects $\tau^k_{\mathbb{C}}$ and $\mathcal{H}_{\mathbb{C}}$ are purely algebraic. That means, we can enlarge the finitely generated subfield $L$ such that the ample invertible sheaf $\mathcal{H}_{\mathbb{C}}$ as well as the morphism $\mathrm{Sym}^{m\eta}T_{Y_{\mathbb{C}}}(-\log S_{\mathbb{C}}) \to \bigoplus \mathcal{H}^{\vee}_{\mathbb{C}}$ are also defined over $L$. 
So the ample invertible sheaf $\mathcal{H}_{\mathbb{C}}$ as well as the map can be base changed to the original compactified base space $Y_K$ over $K$ via the field embedding $L \subset K$.\\
Therefore we obtain the composed map over $K$
\begin{align}\label{sym-diff}
\mathrm{Sym}^{m\eta} T_{Y_K}(-\log S_K) \to \mathrm{Sym}^{\eta}\mathcal{P}^{\vee}_K \to \bigoplus \mathcal{H}^{\vee}_K
\end{align}
where the second map is injective over a Zariski open subset. Note that $\mathcal{H}_K$ is still ample. 
\begin{defi}
We say a rigid analytic morphism $f:\, C^{\mathrm{an}}_K \to U^{\mathrm{an}}_K$ has \emph{maximal length of Griffiths-Yukawa coupling} if the following composed map
\begin{align}\label{max-len}
T^{\otimes m}_{C^{\mathrm{an}}_K} \xrightarrow{df^{\otimes m}} f^*\mathrm{Sym}^{m} T_{Y_K}(-\log S_K) \to f^*\mathcal{P}^{\vee}_K
\end{align}
is nonzero.  
\end{defi}

\begin{prop}\label{max-len-map}
Let $V_K \to U_K$ be a smooth family of polarized varieties with semi-ample canonical divisor. 

Assume that the induced classifying map from $U_K$ to the moduli space is generically finite onto its image. Then any rigid analytic morphism $f:\, C^{\mathrm{an}}_K \to U^{\mathrm{an}}_K$ with maximal length of Griffiths-Yukawa coupling is algebraic.  
\end{prop}
\begin{proof}
Since the composed map (\ref{max-len}) is nonzero, one can find a copy of $\mathcal{H}^{\vee}_K$ in the direct sum appearing in the diagram (\ref{sym-diff}) such that the composition
\[
T^{\otimes m\eta}_{C^{\mathrm{an}}_K} \xrightarrow{df^{\otimes m\eta}} f^*\mathrm{Sym}^{m\eta} T_{Y_K}(-\log S_K) \to f^*\mathcal{H}^{\vee}_K
\]
is nonzero. In this way we have found a symmetric differential $\omega \in \Gamma(Y_K, \mathrm{Sym}^{m\eta} \Omega^1_{Y_K}(\log S_K) \otimes \mathcal{H}^{\vee}_K)$ which vanishes along the ample divisor associated to $\mathcal{H}_K$, and the pull back $f^*\omega \not\equiv 0$. Therefore, by Theorem~\ref{extension}, we can extend $f$ over all the points of  $\bar{C}^{\mathrm{an}}_K \setminus C^{\mathrm{an}}_K$. Now the algebraicity follows from the rigid GAGA theorem {\cite[Theorem~4.10.5]{FvP}}.
\end{proof}

Now we can state our main theorem.

\begin{thm}\label{main-thm}
Let $V_K \to U_K$ be a smooth family of polarized varieties with semi-ample canonical divisor. Suppose that the induced classifying map from $U_K$ to the moduli space is quasi-finite. If the Griffiths-Yukawa coupling length of the induced complex family $V_{\mathbb{C}} \to U_{\mathbb{C}}$ is one, then $U_K$ is $K$-analytically Borel hyperbolic.   
\end{thm}
\begin{proof}
By Theorem~\ref{K-testing}, we consider a rigid analytic morphism $f$ from the rigid analytification of a smooth quasi-projective curve $C^{\mathrm{an}}_K$ to $U^{\mathrm{an}}_K$. By replacing $U_K$ by the Zariski closure of $f(C^{\mathrm{an}}_K)$, one can assume that the image of $f$ is Zariski dense. The family restricted to the new base has maximal variation as we assume that the classifying map is quasi-finite. After desingularizing the base space by Hironaka's theorem, we can assume that $U_K$ is smooth, and the family still has maximal variation. The rigid analytic morphism $f$ can be lifted to the smooth model since its image is not contained in the center of birational modifications.\\
Now since the Griffiths-Yukawa coupling length $m=1$, the diagram (\ref{max-len}) becomes
\[
T_{C^{\mathrm{an}}_K} \xrightarrow{df} f^*T_{Y_K}(-\log S_K) \to f^*\mathcal{P}^{\vee}_K.
\]
Note that the second map factors through from the Kodaira-Spencer map. Since the family has maximal variation, it is generically injective. Combining with the Zariski density of the image of $f$, we know that $f$ has maximal length of Griffiths-Yukawa coupling. Then we apply Proposition~\ref{max-len-map}.
\end{proof}

\begin{proof}[Proof of Theorem~\ref{thm-A}]
For $N \geq 3$, $\mathcal{M}^{[N]}_{g,K}$ is a fine moduli space and so we have a universal family over it.
By {\cite[Theorem~1.4, ii)]{VZ-2}}, the Griffiths-Yukawa coupling length of the family can be bounded from above by its fiber dimension, which is one for family of curves. Now we apply Theorem~\ref{main-thm}.  
\end{proof}
\begin{proof}[Proof of Theorem~\ref{Cor-B}]
By {\cite[Theorem~2.18]{Java-Vezz}}, we know that $\mathcal{M}^{[N]}_{g,K}$ is $K$-analytically Brody hyperbolic if and only if every rigid analytic morphism $\mathbb{G}^{\mathrm{an}}_{m,K} \to \mathcal{M}^{[N],\mathrm{an}}_{g,K}$ is constant and, for every abelian variety $B$ over $K$ with good reduction over $\mathcal{O}_K$, every morphism $B \to \mathcal{M}^{[N]}_{g,K}$ is constant.\\
We first check the statement about abelian varieties. Since one only needs to consider morphisms $B \to \mathcal{M}^{[N]}_{g,K}$, we are able to find some finitely generated subfield such that the morphism is defined over it. After embedding this finitely generated subfield  into $\mathbb{C}$ and the base change of the morphism, we get a morphism $B_{\mathbb{C}} \to \mathcal{M}^{[N]}_{g,\mathbb{C}}$, which has to be constant by the hyperbolicity of $\mathcal{M}^{[N]}_{g,\mathbb{C}}$. This forces the original morphism to be constant.\\
Next we check the statement about rigid analytic morphisms $\mathbb{G}^{\mathrm{an}}_{m,K} \to \mathcal{M}^{[N],\mathrm{an}}_{g,K}$. Since we have already proved that $\mathcal{M}^{[N]}_{g,K}$ is $K$-analytically Borel hyperbolic, every such rigid analytic morphism is actually the rigid analytification of some morphism $\mathbb{G}_{m,K} \to \mathcal{M}^{[N]}_{g,K}$. Then using the descent argument again, we obtain a morphism $\mathbb{G}_{m,\mathbb{C}} \to \mathcal{M}^{[N]}_{g,\mathbb{C}}$, which is forced to be constant by the hyperbolicity of $\mathcal{M}^{[N]}_{g,\mathbb{C}}$.  
\end{proof}
\begin{rmk}
It is worthwhile to mention that over complex numbers the moduli stack $\mathcal{M}_h$ of polarized complex smooth projective varieties with semi-ample canonical divisor and Hilbert polynomial $h$ is proven to be Borel hyperbolic recently in \cite{DLSZ}. One might expect the $K$-analytic Borel hyperbolicity to hold for more general moduli stacks over $K$.    
\end{rmk}

\end{document}